\documentclass[10pt]{amsart}
\usepackage[cp1251]{inputenc}
\usepackage{amstext}
\usepackage{amsfonts}
\usepackage{amssymb}
\usepackage{amsbsy}
\usepackage{amsmath}
\usepackage{latexsym}
\usepackage{xy}
\usepackage{hhline}
\xyoption{all}

\newcounter{num}[section] %

\newenvironment{theo}
{\refstepcounter{num}%
\bigskip\noindent{\bf Theorem~\arabic{section}.\arabic{num}. }\it}

\newenvironment{cor}
{\refstepcounter{num}%
\bigskip\noindent{\bf Corollary~\arabic{section}.\arabic{num}. }\it}

\newenvironment{lemma}
{\refstepcounter{num}%
\bigskip\noindent{\bf Lemma~\arabic{section}.\arabic{num}. }\it}

\newenvironment{example}
{\refstepcounter{num}%
\bigskip\noindent{\bf Example~\arabic{section}.\arabic{num}.}}

\newenvironment{conj}
{\refstepcounter{num}%
\bigskip\noindent{\bf Conjecture~\arabic{section}.\arabic{num}. }\it}

\newenvironment{remark}
{\refstepcounter{num}%
\bigskip\noindent{\bf Remark~\arabic{section}.\arabic{num}.}}

\newcommand{\definition}[1]
{\refstepcounter{num}%
\bigskip\noindent{\bf Definition~\arabic{section}.\arabic{num}}~({\it #1}).}

\newcommand{\Ref}[1]{(\ref{#1})}

\newcounter{thepic}

\newenvironment{eq}{\begin{equation}}{\end{equation}}

\newcommand{\si}{\sigma}
\newcommand{\al}{\alpha}
\newcommand{\be}{\beta}
\newcommand{\ga}{\gamma}

\newcommand{\de}{\delta}

\newcommand{\LA}{\langle}
\newcommand{\RA}{\rangle}

\newcommand{\ov}[1]{\overline{#1}}
\newcommand{\un}[1]{{\underline{#1}} }
\newcommand{\id}[1]{{{\rm id}\{{#1}\}}}

\newcommand{\mdeg}{\mathop{\rm mdeg}}
\newcommand{\Char}{\mathop{\rm char}}

\newcommand{\X}{\LA X\RA}

\newcommand{\FF}{{\mathbb{F}}}   
\newcommand{\LL}{{\mathbb{L}}}   
\newcommand{\NN}{{\mathbb{N}}}
\newcommand{\ZZ}{{\mathbb{Z}}}   

\begin{document}
\renewcommand{\refname}{References}
\thispagestyle{empty}

\title{Associative nil-algebras over finite fields.}%
\author{{Artem A. Lopatin and Ivan P. Shestakov}}%
\noindent
\address{\noindent{}Artem A. Lopatin 
\newline\hphantom{iiii} Omsk Branch of
\newline\hphantom{iiii} Federal State Budgetary Scientific Establishment 
\newline\hphantom{iiii} Sobolev Institute of Mathematics, 
\newline\hphantom{iiii} Siberian Branch of the Russian Academy of Sciences 
\newline\hphantom{iiii} (OB IM SBRAS). 
\newline\hphantom{iiii} Pevtsova street, 13,
\newline\hphantom{iiii} 644043, Omsk, Russia.
\newline\hphantom{iiii} {\it E-mail address:} artem\underline{ }lopatin@yahoo.com%
\newline{}
\newline\hphantom{iiii} Ivan P. Shestakov 
\newline\hphantom{iiii} Institute of Mathematics and Statistics,
\newline\hphantom{iiii} University of S\~ao Paulo,
\newline\hphantom{iiii} Caixa Postal 66281, 
\newline\hphantom{iiii} S\~ao Paulo-SP, Brazil, 05311-970.
\newline\hphantom{iiii} {\it E-mail address:} shestak@ime.usp.br%
}

\vspace{1cm}
\maketitle {\small
\begin{quote}
\noindent{\sc Abstract. }  
The nilpotency degree of a relatively free finitely generated associative algebra with the identity $x^n=0$ is studied over finite fields.
\medskip


\end{quote}
}

\section{Introduction}\label{section_intro}

We denote by $\X_d\subset \X$ the semigroups (without unity) freely generated by {\it letters}  $x_1,\ldots,x_d$ and $x_1,x_2,\ldots$, respectively. Elements of $\X$ are called {\it words}.  Let $\FF\X_{d}$ and $\FF\X$ be the vector spaces over the field $\FF$ with the bases $\X_d$ and $\X$, respectively. Note that elements of $\FF\X_{d}$ and $\FF\X$ are finite linear combinations of words from $\X_d$ and $\X$, respectively. Denote by 
$$N_{n,d}=\frac{\FF\X_d}{\id{x^n\,|\,x\in\FF\X_d}}$$%
the relatively free finitely generated associative algebra with the identity $x^n=0$.  The connection between this algebra and analogues of the Burnside problems for associative algebras suggested by Kurosh and Levitzky is discussed in recent survey~\cite{Zelmanov07} by Zelmanov. The algebra $N_{n,d}$ also plays a crucial role in the construction of minimal systems of generators for algebras of polynomial invariants of several matrices  (see~\cite{DKZ02},~\cite{Lopatin_Comm2},~\cite{Lopatin_2222}, \cite{Lopatin_2222_II}, \cite{Lopatin_O3}). It is well-known that $N_{n,d}$ is a nilpotent algebra. For example, it follows from the Shirshov Height Theorem~\cite{Shirshov57} and the fact that 
\begin{eq}\label{eq_intro_1}
L_{1^n}(a_1,\ldots,a_n)=\sum_{\si\in S_n} a_{\si(1)}\cdots a_{\si(n)},
\end{eq}%
the complete linearization of $x^n$, is equal to zero in $N_{n,d}$ (see below for the details). We write 
$$C_{n,d}=\min\{c>0\,|\,a_1\cdots a_c=0 \text{ for all }a_1,\ldots,a_c\in N_{n,d}\}$$ 
for the {\it nilpotency} degree of $N_{n,d}$. Since $C_{1,d}=1$ and $C_{n,1}=n$, we assume that $n,d\geq 2$ unless otherwise stated. To specify the field $\FF$, we write $N_{n,d}^{\FF}$, $C_{n,d}^{\FF}$ for  $N_{n,d}$, $C_{n,d}$, respectively. We say that an element $f$ of $\FF\X_d$ is a {\it relation} for $N_{n,d}$ and write $f=0$ in $N_{n,d}$ if the image of $f$ in $N_{n,d}$ is zero.

In characteristic zero case of the field $\FF$ we have $\frac{1}{2}n(n+1)\leq C_{n,d}\leq n^2$, where the lower bound was established by Kuzmin~\cite{Kuzmin75} and the upper bound was given by Razmyslov~\cite{Razmyslov74}. Kuzmin also conjectured that $C_{n,d}=\frac{1}{2}n(n+1)$, which was shown to be true for $n\leq4$ by Vaughan--Lee~\cite{Vaughan93}. The case of $n=5$ and $d=2$ was considered by Shestakov and Zhukavets~\cite{Shestakov04}. An English translation of Kuzmin's result can be found in books~\cite{Belov_book05} and~\cite{Drensky_book04}.

The difference of the case of positive characteristic of $\FF$ from the case of zero characteristic is due to the fact that in the latter case any relation for $N_{n,d}$ belongs to the ideal generated by $L_{1^n}$, which in general is not the case in positive characteristic. As the result, in case $0<\Char{\FF}\leq n$ we have $C_{n,d}\to\infty$ as $d\to\infty$ by~\cite{DKZ02}. Recently, Belov and Kharitonov~\cite{Belov11} established the that $C_{n,d}\leq 2^{18}\cdot n^{12\log_3(n)+28}d$. In the case of an infinite field $\FF$ with $\Char{\FF}>\frac{n}{2}$ Lopatin~\cite{Lopatin_Nnd} proved that $C_{n,d}<4\cdot 2^{\frac{n}{2}}d$. See Remark~4.8 of~\cite{Lopatin_Nnd} for the comparison of these two upper bounds. 

Given a field $\FF$ of characteristic $p\geq0$, the nilpotency degree $C_{n,d}$ is known for $n=2$ for an arbitrary $\FF$ (for example, see~\cite{DKZ02}) and for $n=3$ for an infinite $\FF$ (see~\cite{Lopatin_Comm1} and~\cite{Lopatin_Comm2}):
$$
C_{2,d}=\left\{
\begin{array}{rl}
3,&\text{if } p=0 \text{ or }p>2\\
d+1,&\text{if } p=2 \\
\end{array}
\right.
\;\text{ and }\;
C_{3,d}=\left\{
\begin{array}{rl}
6,&\text{if } p=0 \text{ or }p>3\\
6,&\text{if } p=2 \text{ and }d=2\\
d+3,&\text{if } p=2 \text{ and } d>2\\
3d+1,&\text{if }p=3.\\
\end{array}
\right..
$$
Moreover, in case $\FF$ is infinite Lopatin explicitly described an $\FF$-basis for $N_{3,d}$ (see~\cite{Lopatin_Comm2}) and calculated $C_{4,d}$ with deviation three for all $d$ under assumption that $p\neq2$ (see~\cite{Lopatin_Nnd}). 

In the case of an infinite field $\FF$ all partial linearizations of $x^n$ are relations for $N_{n,d}$, which in general does not hold in the case of a finite field. For an arbitrary field of characteristic $p$ Eick~\cite{Eick11} obtained the following results:
\begin{enumerate}
\item[$\bullet$] if $p>n$, then $C_{n,d}$ is the same as in the case of an infinite field of the same characteristic; 

\item[$\bullet$] $C_{3,d}$ is computed for $p=2,3$ and $d\leq4$;

\item[$\bullet$] $C_{4,2}$ is calculated for $p=2,3,5$ and $C_{5,2}$ is calculated for $p=0,2,3,5,7$.
\end{enumerate}

In this paper we extend the mentioned results as follows. Let $\FF$ be an arbitrary field and $\LL$ be an infinite field with $\Char{\FF}=\Char{\LL}$. Then 
\begin{enumerate}
\item[$\bullet$] in case $\#\FF\geq n$ we have $C_{n,d}^{\FF}=C_{n,d}^{\LL}$ (see Corollary~\ref{cor1});

\item[$\bullet$] we completed the description of $C_{3,d}^{\FF}$ for all $d$ by proving that $C_{3,d}^{\FF}=C_{3,d}^{\LL}$ (see Section~\ref{section_n3}). We also explicitly described an $\FF$-basis for $N_{3,d}^{\FF}$ (see Remark~\ref{remark_n3}). Note that 
$$
\dim{N_{3,d}^{\FF}}=\left\{
\begin{array}{rl}
\dim{N_{3,d}^{\LL}}+d(d-1)/2,&\text{if } \#\FF=2\\
\dim{N_{3,d}^{\LL}},&\text{otherwise }\\
\end{array}
\right.;$$ 

\item[$\bullet$] in case $\#\FF=n-1$ we have  $C_{n,d}^{\LL} \leq C_{n,d}^{\FF}\leq C_{n,d}^{\LL} + 1$ (see Theorem~\ref{theo_sec4}); 

\item[$\bullet$] as a consequence, $C_{4,d}$ is described with deviation four for all $d$ under assumption that $p\neq2$ (see Corollary~\ref{cor_n4}).
\end{enumerate}

The following conjecture holds for $n=2,3$ and all $d$, for $n=4,5$ and $d=2$.

\begin{conj}\label{conj1}
For a finite field $\FF$ and an infinite field $\LL$ with $\Char{\FF}=\Char{\LL}$ the equality $C_{n,d}^{\FF}=C_{n,d}^{\LL}$ holds.
\end{conj}

\section{General case}\label{section_general}

We start with some notations. Let $\NN=\{1,2,\ldots\}$, $\NN_0=\NN\sqcup\{0\}$, $\ZZ$ be the ring of integers, and $\FF^{\ast}=\FF\backslash \{0\}$. For an $a\in\X_d$ and a letter $x$ we denote by $\deg_x(a)$ the degree of $a$ in the letter $x$, by $\mdeg(a)=(\deg_{x_1}(a),\ldots,\deg_{x_d}(a))$ the {\it multidegree} of $a$, and by $\deg(a)=\deg_{x_1}(a)+\cdots+\deg_{x_d}(a)$ the {\it degree} of $a$. As usually, the degree of $f=\sum_i\al_i a_i\in\FF\X_d$ is the maximum of degrees of words $a_i$. For short, we write $1^n$ for multidegree $(1,\ldots,1)$ ($n$ times). Given $\un{\al}=(\al_1,\ldots,\al_r)\in\NN_{0}^r$, we set $\#\un{\al}=r$, $|\un{\al}|=\al_1+\cdots+\al_r$. Given a prime $p$, denote by $\FF_p$ the field with $p$ entries. In case $p=0$ we set $\FF_p=\ZZ$.

We define $\X_d^{\#}=\X_d\sqcup\{1\}$ and $\X^{\#}=\X\sqcup\{1\}$. Given $\un{\theta}\in\NN_{0}^r$ with $|\un{\theta}|=n$, denote by $L_{\un{\theta}}(\un{x})=L_{\un{\theta}}(x_1,\ldots,x_r)\in\FF\X$ the coefficient of $\al_1^{\theta_1}\cdots \al_r^{\theta_r}$ in $(\al_1 x_1 +\cdots + \al_r x_r)^n$, where $\al_i\in\FF$. Note that $L_{\un{\theta}}(\un{x})$ is a non-zero element of $\FF\X$.  We say that $L_{\un{\theta}}(\un{x})$ is the {\it partial linearization} of $x^n$ of multidegree $\un{\theta}$. As an example, see above formula~\Ref{eq_intro_1} for $L_{1^n}$.

\begin{remark}\label{remark0}
If $\FF$ is infinite, then by the standard Vandermonde arguments (or see the proof of part~(a) of Theorem~\ref{theo1} below) we obtain that the ideal of relations for $N_{n,d}$  is generated by $L_{\un{\theta}}(\un{a})$ for $|\un{\theta}|=n$ and $a_i\in\X_d$ for all $i$.
\end{remark}

\begin{remark}\label{remark_important}
Assume that the characteristic of $\FF$ is $p\geq0$.  Let $A=\FF\X_d/I$ for the ideal $I$ generated by some polynomials from $\FF_p\X_d$. Consider a basis $B$ of $A$ consisting of words. Then in the algebra $A$ we have
\begin{enumerate}
 \item[(a)] every $w\in\X_d$ belongs to the $\FF_p$-span of $B$ in $A$;

 \item[(b)] if $w_1,\ldots,w_r\in\X_d$ are linearly dependent in $A$, then $\sum_i \al_i w_i=0$ in $A$ for such $\al_i\in\FF_p$ that not all of them are equal to zero.
\end{enumerate}
\end{remark}
\bigskip

Remarks~\ref{remark0} and~\ref{remark_important} imply that if $\FF$ and $\LL$ are infinite fields and $\Char{\FF}=\Char{\LL}$, then 
\begin{enumerate}
\item[$\bullet$] every basis for $N_{n,d}^{\FF}$ consisting of words is a basis for $N_{n,d}^{\LL}$;

\item[$\bullet$] $C_{n,d}^{\FF}=C_{n,d}^{\LL}$.
\end{enumerate}%

\definition{of Frobenius partial linearization}\label{def_Frobenius} Assume that $\FF$ is finite and $q=\#\FF$. Given $\un{\de}\in\NN^r$ with $|\un{\de}|=n$, we say that the {\it Frobenius partial linearization} of $x^n$ of multidegree $\un{\delta}$ is 
$$F_{\un{\delta}}(\un{x})=F_{\un{\delta}}(x_1,\ldots,x_r)=\sum L_{\un{\theta}}(x_1,\ldots,x_r)\in\FF\X,$$
where $\un{\theta}$ ranges over vectors from $\NN^r$ satisfying 
\begin{enumerate}
\item[$\bullet$] $|\un{\theta}|=n$;

\item[$\bullet$] $\theta_i\equiv\delta_i \;({\rm mod}\; q-1)$ for all $i$.
\end{enumerate} 
\medskip

As an example, if $n=5$ and $\#\FF=3$, then $F_{41}(\un{x})=L_{41}(\un{x})+L_{23}(\un{x})$, $F_{32}(\un{x})=L_{32}(\un{x})+L_{14}(\un{x})$,  and $F_{311}(\un{x})=L_{311}(\un{x})+L_{131}(\un{x})+L_{113}(\un{x})$. Note that $L_{\un{\theta}}(x_1,\ldots,x_r)=L_{\si\un{\theta}}(x_{\si(1)},\ldots,x_{\si(r)})$ for all permutations $\si\in S_r$, where $\si\un{\theta}$ stands for $(\theta_{\si(1)},\ldots,\theta_{\si(r)})$. Similar remark holds for $F_{\un{\de}}(\un{x})$. We will use these remarks without references to them.

A vector $\un{\theta}\in\NN^r$ is called {\it ordered} if $\theta_1\geq \cdots \geq \theta_r$. A vector $\un{\de}\in\NN^r$ is called {\it $q$-maximal} if $\un{\de}\geq \un{\theta}$ for every $\un{\theta}\in\NN^r$ with $|\un{\theta}|=|\un{\de}|$ and $\theta_i\equiv\delta_i \;({\rm mod}\; q-1)$ for all $i$, where $\geq$ stands for the usual lexicographical order on $\NN^r$.

\begin{remark}\label{remark1} For $q=\#\FF$ we have that 
\begin{enumerate}
\item[(a)] if $q\geq n$, then $F_{\un{\delta}}(\un{x})=L_{\un{\delta}}(\un{x})$;

\item[(b)] every word $a\in\X_r$ of degree $n$, satisfying $\deg_{x_i}(a)>0$ for all $1\leq i\leq r$, is a summand of $F_{\un{\de}}(\un{x})$ for one and only one $q$-maximal vector $\un{\de}\in\NN^r$ with $|\un{\de}|=n$;

\item[(c)] for every $\un{\de}\in \NN^r$ there exists a $q$-maximal $\un{\nu}\in\NN^r$ satisfying $F_{\un{\de}}(\un{x})=F_{\un{\nu}}(\un{x})$ and $|\un{\de}|=|\un{\nu}|$. 
\end{enumerate}
\end{remark}

\begin{theo}\label{theo1}
Assume $\FF$ is finite and $q=\#\FF$. Then the ideal of relations for $N_{n,d}$ is generated as a vector space over $\FF$ by $uf(a_1,\ldots,a_r)v$ for all $a_1,\ldots,a_r\in\X_d$, $u,v\in\X_d^{\#}$ and $f\in S$, where the set $S$ is defined as follows:
\begin{enumerate}
\item[(a)] if $q\geq n$, then $S$ consists of $L_{\un{\theta}}(\un{x})$ for all ordered $\un{\theta}$ with $|\un{\theta}|=n$; 

\item[(b)] if $q<n$, then $S$ consists of $F_{\un{\de}}(\un{x})$ for all ordered $q$-maximal $\un{\de}$ with $|\un{\de}|=n$.
\end{enumerate} 
\end{theo}
\bigskip

We split the proof of the theorem into several lemmas. Given $1\leq r\leq n$, denote 
$$P_r(\un{x})=P_r(x_1,\ldots,x_r)=\sum L_{\un{\theta}}(x_1,\ldots,x_r),$$
where the sum ranges over all $\un{\theta}\in\NN^r$ with $|\un{\theta}|=n$.

\begin{lemma}\label{lemmaP}
We have $P_r(a_1,\ldots,a_r)=0$ in $N_{n,d}$ for all $1\leq r\leq n$ and $a_1,\ldots,a_r\in\FF\X_d$.
\end{lemma} 
\begin{proof}
We prove by the inducton on $r$. If $r=1$, then $P_1(a_1)=a_1^n=0$ in $N_{n,d}$ for all $a_1\in\FF\X_d$.

Assume that for all $1\leq k<r$ we have $P_k(a_1,\ldots,a_k)=0$ in $N_{n,d}$ for all $a_1,\ldots,a_k\in\FF\X_d$. Since $(a_1+\cdots+a_r)^n$ is equal to
$$P_r(a_1,\ldots,a_r)+\sum_{1\leq k\leq r-1}\sum P_k(a_{i_1},\ldots,a_{i_k}),$$
where the second sum ranges over all $1\leq i_1<\cdots<i_k\leq r$, the induction hypothesis completes the proof.
\end{proof}

\begin{lemma}\label{lemmaR}
Let $\FF$ be finite and $q=\#\FF$. Then $F_{\un{\delta}}(\un{a})=0$ in $N_{n,d}$ for every $\un{\delta}$ with $|\un{\de}|=n$ and $a_i\in\FF\X_d$ for all $i$. In particular, in case $q\geq n$ we have $L_{\un{\theta}}(\un{a})=0$ in $N_{n,d}$ for all $\un{\theta}$ with $|\un{\theta}|=n$ and $a_i\in\FF\X_d$. 
\end{lemma}
\begin{proof}
By part (c) of Remark~\ref{remark1}, without loss of generality we can assume that $\un{\de}\in\NN^r$ is a $q$-maximal vector, where $r>0$. Let $\{\un{\de}=\un{\de}^{1},\ldots,\un{\de}^{k}\}$ be the set of all pairwise different $q$-maximal vectors of $\NN^r$ with $|\un{\de}^{i}|=n$. By part~(b) of Remark~\ref{remark1}, we obtain
$$P_r(\un{x})=F_{\un{\de}^{1}}(\un{x})+\cdots+F_{\un{\de}^{k}}(\un{x}).$$
Consider vectors $\un{\nu}^1,\ldots,\un{\nu}^k$ from $\NN_0^r$ such that $0\leq \nu_i^j\leq q-2$ and $\nu_i^j\equiv \de_i^j \;({\rm mod}\; q-1)$ for all $i$ and $j$. By the equality $\al^{q-1}=\al^0$ in $\FF$, for $\al_1,\ldots,\al_r$ from $\FF$ we have 
$$P_r(\al_1 a_1,\ldots,\al_r a_r) = \un{\al}^{\un{\nu}^{1}}F_{\un{\de}^{1}}(a_1,\ldots,a_r)  + \cdots + \un{\al}^{\un{\nu}^{k}}F_{\un{\de}^{k}}(a_1,\ldots,a_r)$$
for all $a_1,\ldots,a_r\in\FF\X_d$, where $\un{\al}^{\un{\nu}^{j}}$ stands for $\al_1^{\nu^j_1}\cdots\al_r^{\nu^j_r}$.
Denote by $(E_{\un{\al}})$ the following linear equation in variables $y_1,\ldots,y_k$: 
$$(E_{\un{\al}}): \qquad\qquad %
\un{\al}^{\un{\nu}^{1}} y_1  + \cdots + \un{\al}^{\un{\nu}^{k}} y_k=0.$$
Consider the system of all linear equations $(E_{\un{\al}})$, 
where $\al_1,\ldots,\al_r$ range over the set of non-zero elements of $\FF$. 

Consider some non-zero elements $\al_2,\ldots,\al_r$ from $\FF$. For $0\leq s\leq q-2$ denote   
$$z_s=z_s(\al_2,\ldots,\al_r)=\sum\nolimits_j \al_2^{\nu^j_2}\cdots \al_r^{\nu^j_r} y_j,$$
where the sum ranges over all $1\leq j\leq k$ with $\nu^j_1=s$. In case $\nu^j_1 \neq s$ for all $j$ we set $z_s=0$ and say that $z_s$ is trivial. Since $0\leq\nu^{j}_1 \leq q-2$ for all $j$, the equation $z_0+\al_1 z_1+\al_1^2 z_2+\cdots+\al_1^{q-2} z_{q-2}=0$ holds for all $\al_1\in\FF$. Denote non-zero elements of $\FF$ by $\xi_1,\ldots,\xi_{q-1}$. Therefore, $B\cdot (z_0,\ldots,z_{q-2})^T=0$ for the Vandermonde matrix  
$$B=
\left(\begin{array}{ccccc}
1 & \xi_1 & \xi_1^2 & \cdots & \xi_1^{q-2} \\
1 & \xi_2 & \xi_2^2 & \cdots & \xi_2^{q-2} \\
\vdots & \vdots & \vdots & & \vdots \\
1 & \xi_{q-1} & \xi_{q-1}^2 & \cdots & \xi_{q-1}^{q-2} \\       
\end{array}
\right).
$$
Then $\det(B)=\prod_{1\leq i<j<q} (\xi_j-\xi_i)$ is not zero and $z_0=\cdots=z_{q-2}=0$ for all $\al_2,\ldots,\al_r$. 

Then for every non-trivial $z_s$ we consider the system of linear equations $z_s(\al_2,\ldots,\al_r)=0$,  where $\al_2,\ldots,\al_r$ range over the set of non-zero elements of $\FF$. 
Note that for every $1\leq j\leq k$ the variable $y_j$ appears in one and only one such system of linear equations. Repeating the above reasoning for $\al_2$, $\al_3$ and so on, we finally obtain that $y_1=\cdots=y_k=0$.  Lemma~\ref{lemmaP} concludes the proof of the first part of the lemma. The second part follows from part~(a) of Remark~\ref{remark1}.
\end{proof}


Now we can proof Theorem~\ref{theo1}.

\begin{proof}
By part~(b) of Remark~\ref{remark1}, for $a_1,\ldots,a_k$ from $\X_d$ and $\al_1,\ldots,\al_k$ from $\FF$ we have
$$(\al_1 a_1 + \cdots + \al_k a_k)^n = 
\sum_{r=1}^{\min\{n,k\}} \sum_{\un{\de}} \sum_{\varphi} F_{\un{\de}} (\al_{\varphi(1)} a_{\varphi(1)},\ldots,\al_{\varphi(k)} a_{\varphi(k)}),$$
where $\un{\de}$ ranges over $q$-maximal vectors from $\NN^r$ with $|\un{\de}|=n$ and $\varphi$ ranges over strictly increasing maps $\ov{1,r}\to \ov{1,k}$. Since 
$$F_{\un{\de}}(\al_1 a_1,\ldots,\al_r a_r) = \al_1^{\de_1}\cdots \al_r^{\de_r} F_{\un{\de}}(a_1,\ldots,a_r),$$%
Lemma~\ref{lemmaR} implies that the ideal of relations for $N_{n,d}$ is generated by all $F_{\un{\de}}(\un{a})$ with $q$-maximal $\un{\de}$ satisfying $|\un{\de}|=n$ and $a_i\in\X_d$. Parts~(a) and~(c) of Remark~\ref{remark1} complete the proof of parts (a) and (b) of the theorem, repectively.
\end{proof}

\begin{example}\label{ex1} In the formulation of Theorem~\ref{theo1} the set $S$ is the following one:
\begin{enumerate}
\item[$\bullet$] if $\#\FF=n-1$, then $S$ consists of $L_{n-1,1} +L_{1,n-1}$ and $L_{\un{\theta}}$ for all ordered $\un{\theta}$ with $|\un{\theta}|=n$ and $\un{\theta}\neq (n-1,1)$;  

\item[$\bullet$] if $\#\FF=2$, then $S=\{P_1,\ldots,P_n\}$;   

\item[$\bullet$] if $n=5$ and $\#\FF=3$, then $S$ consists of $L_5$, $L_{41}+L_{23}$, $L_{311}+L_{131}+L_{113}$, $L_{221}$, $L_{2111}$, $L_{1^5}$.  
\end{enumerate}
\end{example}

\begin{cor}\label{cor1}
Assume $\FF$ is a finite field and $\LL$ is a finite or infinite field with $\Char{\FF}=\Char{\LL}=p$ and $\#\FF<\#\LL$. 
\begin{enumerate}
\item[1.] Let $\#\FF<n$ and $\#\LL<n$. If $l|t$ for $\#\FF=p^l$ and $\#\LL=p^t$, then $C_{n,d}^{\FF}\geq C_{n,d}^{\LL}$.

\item[2.] If $\#\FF<n\leq \#\LL$, then $C_{n,d}^{\FF}\geq C_{n,d}^{\LL}$.

\item[3.] If $\#\FF\geq n$ and $\#\LL\geq n$, then $C_{n,d}^{\FF}= C_{n,d}^{\LL}$.
\end{enumerate}
\end{cor}
\begin{proof}
In part~1 we can assume that $\FF\subset \LL$, which implies the required. Consider parts~2 and 3 of the theorem. 
Theorem~\ref{theo1} and Remark~\ref{remark0} imply that the ideal of relations for $N_{n,d}^{\FF}$ ($N_{n,d}^{\LL}$, respectively) is generated by some set $W^{\FF}\subset\FF_p\X_d$ (a set $W^{\LL}\subset\FF_p\X_d$, respectively). Moreover, $W^{\FF}$ belongs to $\FF_p$-span of $W^{\LL}$ in part~2 and $W^{\FF}=W^{\LL}$ in part~3. The proof is completed.
\end{proof}

\section{The case of $x^3=0$}\label{section_n3}

In this section we investigate $N_{3,d}^{\FF}$ for any field $\FF$.  In the case of an infinite field $\FF$, the nilpotency degree and a basis for $N_{3,d}^{\FF}$ were established in~\cite{Lopatin_Comm2} developing ideas from~\cite{Lopatin_Comm1}. By Theorem~\ref{theo1} and Remarks~\ref{remark0} and~\ref{remark_important}, in the case of $\#\FF>2$, the nilpotency degree and a basis consisting of words are the same as in the case of infinite field of the same characteristic. So the only case that is left to consider is the case of $\FF=\FF_2$.

\begin{theo}\label{theo_n3}
Let $\FF=\FF_2$. Then for every infinite field $\LL$ of characteristic $2$ we have
\begin{enumerate}
 \item[(a)] 
$C_{3,d}^{\FF}=C_{3,d}^{\LL}=\left\{
\begin{array}{rl}
6,& \text{ if }d=2\\
d+3,& \text{ if }d\geq3\\ 
\end{array} 
\right.;$ 
\item[(b)] $\dim{N_{3,d}^{\FF}}=\dim{N_{3,d}^{\LL}}+d(d-1)/2$.
\end{enumerate}
\end{theo}
\bigskip

The proof of this theorem is given at the end of the section.

\definition{of $\approx$} Given an $f\in\FF\X_d$,  we write $f\approx0$ in $N_{n,d}$ if and only if either $f=0$ in $N_{n,d}$ or $f=\sum_i\al_iw_i$ in $N_{n,d}$ for $\al_i\in\FF$ and $w_i\in\X$ satisfying $\deg(w_i)>\deg(f)$.

\begin{remark}\label{remark_approx}
If $f\approx0$ in $N_{n,d}$, then $ufv\approx0$ in $N_{n,d}$ for all $u,v\in\X_d^{\#}$. On the other hand, note that in case $f\approx0$ and $h\approx0$ in $N_{n,d}$ we can have that $f+h\not\approx0$ in $N_{n,d}$. Similarly, in case $f\approx0$ in $N_{n,d}$ and $f=h$ in $N_{n,d}$ we can have that $h\not\approx0$ in $N_{n,d}$. 
\end{remark}
\bigskip
 
We will use the following remark to obtain relations for $N_{n,d}$.

\begin{remark}\label{remark_inversion}
Define the inversion $\ast$ on $\FF\X$ as follows: $x_i^{\ast}=x_i$ for all $i$, $(ab)^{\ast}=b^{\ast}a^{\ast}$ for all $a,b\in\X$ and $\ast:\FF\X\to\FF\X$ is a linear map. Then for any relation $f\in\FF\X_d$ for $N_{n,d}$ we have that $f^{\ast}$ is also a relation for $N_{n,d}$.
\end{remark}
\bigskip

In the rest of this section we assume that $\FF=\FF_2$ unless otherwise stated. By Theorem~\ref{theo1} and Example~\ref{ex1}, we have 
\begin{eq}\label{eq_L111}
L_{111}(x,y,z)=0 \text{ in }N_{3,d}, 
\end{eq}
\vspace{-0.5cm}
\begin{eq}\label{eq_L21_L12}
L_{21}(x,y)+L_{12}(x,y)=0 \text{ in }N_{3,d}.
\end{eq}%
Moreover, the ideal of relations for $N_{3,d}$ is generated by $x^3$ and left hand sides of equalities~\Ref{eq_L111} and~\Ref{eq_L21_L12} for $x,y,z\in\X_d$. Note that $L_{111}(x,y,z)=xyz+xzy+yxz+yzx+zxy+zyx$ and $L_{21}(x,y)=x^2y+xyx+yx^2$. By the straightforward calculation we can see that  
$$x L_{21}(y,z)= L_{21}(y,xz) + L_{111}(x,y,yz) - y L_{111}(x,y,z)$$%
holds in $\FF\X$. Then equality~\Ref{eq_L111} implies that
\begin{eq}\label{eq_star1}
x L_{21}(y,z)= L_{21}(y,xz) \text{ in }N_{3,d}.
\end{eq}%
By Remark~\ref{remark_inversion}, we have
\begin{eq}\label{eq_star2}
L_{21}(x,y)z= L_{21}(x,yz) \text{ in }N_{3,d}.
\end{eq}

\begin{lemma}\label{lemma_n3_2}
We have $L_{21}(a,b)\approx0$ in $N_{3,d}$ for all $a,b\in\X_d$ with $\deg(a)>1$.
\end{lemma}
\begin{proof} In this proof we work in $N_{3,d}$. 
It is convenient to rewrite equalities~\Ref{eq_star1} and~\Ref{eq_star2} as follows
\begin{eq}\label{eq_star1_prime}
xyzy = xy^2z + xzy^2 + L_{21}(y,xz),
\end{eq}%
\vspace{-0.5cm}
\begin{eq}\label{eq_star2_prime}
xyxz = x^2yz + yx^2z + L_{21}(x,yz),
\end{eq}%
respectively. Making a substitution $z\to yz$ in~\Ref{eq_star2_prime}, we obtain
$$xyxyz = x^2 y^2 z + y x^2 y z + L_{21}(x,y^2 z).$$
Application of~\Ref{eq_star2_prime} to $yx^2yz$ gives us 
\begin{eq}\label{eq_star4}
xyxyz=y^2x^2z + L_{21}(x,y^2z) + L_{21}(y,x^2z).
\end{eq}%
By Remark~\ref{remark_inversion}, we have
\begin{eq}\label{eq_star4_prime}
zxyxy=zy^2x^2 + L_{21}(y,zx^2) + L_{21}(x,zy^2).
\end{eq}%
Making substitutions $y\to yz$ and $z\to y$ in~\Ref{eq_star2_prime}, we obtain
$$xyzxy = x^2 yzy + yzx^2y + L_{21}(x,yzy).$$
Application of~\Ref{eq_star1_prime} to $x^2 yzy$ and the equality $yzx^2 y = y^2 z x^2 + zx^2y^2+L_{21}(y,zx^2)$ in $\FF\X$ gives us 
\begin{eq}\label{eq_star5}
xyzxy = x^2 y^2 z + x^2 z y^2 + z x^2 y^2 + y^2 z x^2  + L_{21}(y,x^2z) + L_{21}(y,zx^2) + L_{21}(x,yzy).
\end{eq}%
Consider $L_{21}(xy,z)=xyxyz + zxyxy + xyzxy$. Applying~\Ref{eq_star4},~\Ref{eq_star4_prime},~\Ref{eq_star5} to the first, second and third summands of $L_{21}(xy,z)$, respectively, we obtain 
$$\begin{array}{rl}
L_{21}(xy,z) & = L_{111}(x^2,y^2,z) + L_{21}(x,y^2z) + L_{21}(y,x^2z) + L_{21}(x,zy^2)\\
& + L_{21}(y,zx^2) + L_{21}(y,x^2z) + L_{21}(y,zx^2) + L_{21}(x,yzy).\\
\end{array}$$
By equalities~\Ref{eq_L111} and~\Ref{eq_L21_L12}, $L_{21}(xy,z)=f(x,y,z)$ for
$$\begin{array}{rl}
f(x,y,z) & = L_{12}(x,y^2z) + L_{12}(y,x^2z) + L_{12}(x,zy^2) + L_{12}(y,zx^2) \\
& + L_{12}(y,x^2z) + L_{12}(y,zx^2) + L_{12}(x,yzy).\\
\end{array}$$
 
Consider $a,b$ from the formulation of the lemma. Let us recall that $\deg(a)>1$. We set $a=xy$ for $x,y\in\X_d$ such that 
\begin{enumerate}
\item[$\bullet$] if $\deg(a)$ is even, then $\deg(x)=\deg(y)=r\geq1$;  

\item[$\bullet$] if $\deg(a)$ is odd, then $\deg(x)=r+1$ and $\deg(y)=r\geq1$.
\end{enumerate}
We have $L_{21}(a,b)=f(x,y,b)$. Note that $f(x,y,b)$ is a sum of words of degrees $k_1=\deg(x)+4\deg(y)+2\deg(b)$ and $k_2=4\deg(x)+\deg(y)+2\deg(b)$. For $k_0=\deg(L_{21}(a,b))=2\deg(x)+2\deg(y)+\deg(b)$ we claim that
$$k_0<k_1 \text{ and } k_0<k_2.$$
For short, we write $s$ for $\deg(b)$. In the case of even $\deg(a)$ we have $k_0=4r+s$, $k_1=k_2=5r+2s$ and in the case of odd $\deg(a)$ we have $k_0=4r+s+2$, $k_1=5r+2s+1$, $k_2=5r+2s+4$. Hence, the claim is proven and $L_{21}(a,b)\approx0$.  
\end{proof}

\begin{lemma}\label{lemma_n3_3}
We have $uL_{21}(a,b)v\approx0$ in $N_{3,d}$ for all $a,b\in\X_d$, $u,v\in\X_d^{\#}$ with $\deg(ua^2bv)\geq4$.
\end{lemma}
\begin{proof}
Assume that $u=v=1$. In case $\deg(a)>1$ the required statement follows from Lemma~\ref{lemma_n3_2}. Let $\deg(a)=1$ and $\deg(b)\geq2$. Then equality~\Ref{eq_L21_L12} gives us the required stetement. 

Assume that $\deg(u)>0$. Since the degree of the right hand side of formula~\Ref{eq_star1} is equal to the degree of its left hand side, formula~\Ref{eq_star1} completes the proof. 
\end{proof}


Now we can proof Theorem~\ref{theo_n3}.

\begin{proof}
Denote by $V_{r}^{\FF}$ the $\FF$-span of all words of $\X_d$ of degree $r$. We write $U_r^{\FF}$ for the image of $V_r^{\FF}$ in $N_{3,d}^{\FF}$.
Note that if a summand of relation~\Ref{eq_L21_L12} has degree three, then the rest of summands of this relation have degree three. Therefore, $N_{3,d}^{\FF}=U_3^{\FF}\oplus W^{\FF}$ for $W^{\FF}=\sum_{r>3} U_r^{\FF}$. 

Let $B=\{b_1,\ldots,b_s\}$ be an $\LL$-basis for $N_{3,d}^{\LL}$ consisting of words. Denote by $B_r$ the set of all $b\in B$ of degree $r$. We claim that 
\begin{eq}\label{claim_n3}
\text{every } w\in\X_d \text{ with } \deg(w)\geq4 \text{ belongs to } \FF\text{-span of } \sqcup_{r\geq4} B_r \text{ in } N_{3,d}^{\FF}. 
\end{eq}%
We set $\deg(w)=r\geq 4$. Since $N_{3,d}^{\LL}$ is homogeneous with respect to the degree, we have $w=\sum_i\be_i b_i$ in $N_{3,d}^{\LL}$ for some $\be_i\in \FF$ and $b_i\in B_r$ (see also Remarks~\ref{remark0} and~\ref{remark_important}). Thus relation~\Ref{eq_L111} together with Lemma~\ref{lemma_n3_3} imply that $w-\sum_i\be_i b_i\approx0$ in $N_{3,d}^{\FF}$. Thus, $w=\sum_i\be_i b_i + \sum_j\ga_j w_j$ for some $\ga_j\in\FF$ and $w_j\in\X_d$ with $\deg(w_j)>r$. Then we consider $w_j$ and so on. Since $N_{3,d}^{\FF}$ is nilpotent, this process will stop at some step. The claim is proven. Similarly to part~(b) of Remark~\ref{remark_important}, we can see that the set $\sqcup_{r\geq4} B_r$ is linearly independent in $N_{3,d}^{\FF}$. Thus, $\sqcup_{r\geq4} B_r$ is a basis for $W^{\FF}$. 

Denote by $B_{111}$ the set of all $x_i x_j x_k$ with pairwise different $i,j,k\in\{1,\ldots,d\}$ such that $(i,j,k)$ is not ordered and denote by $B_{21}$ the set of words $x_i^2x_j$ with $1\leq i\neq j\leq d$.  It is not difficult to see that the set $B_{111}\sqcup B_{21}$ is a basis for $U_3^{\LL}$; and $B_{111}\sqcup B_{21}$ together with words $x_i x_j x_i$, where $1\leq i<j\leq d$, is a basis for $U_3^{\FF}$. The theorem is proven. 
\end{proof}

The next remark follows from the proof of Theorem~\ref{theo_n3}.

\begin{remark}\label{remark_n3}
Let $\FF=\FF_2$ and $\LL$ be an infinite field of characteristic two. If we add elements $x_i x_j x_i$, $1\leq i<j\leq d$, to the basis for $N_{3,d}^{\LL}$ from Theorem~2 of~\cite{Lopatin_Comm2}, then we obtain a basis for $N_{3,d}^{\FF}$. 
\end{remark}

\section{The case of $\#\FF=n-1$}\label{section_sec4}

Let us remark that the case of $\#\FF\geq n$ is considered in Corollary~\ref{cor1}. In this section we investigate the case of field $\FF$ with $n-1$ elements.

\begin{theo}\label{theo_sec4}
Let $\#\FF=n-1$. Then for every infinite field $\LL$ with $\Char{\LL}=\Char{\FF}$ we have  $C_{n,d}^{\LL} \leq C_{n,d}^{\FF}\leq C_{n,d}^{\LL} + 1$. 
\end{theo}
\bigskip

We split the proof of the theorem into several lemmas. The following lemma is a modification of Lemma~1 of Chapter~6 of~\cite{Shestakov_book}.

\begin{lemma}\label{lemma_sec4_1}
For every field $\FF$ we have 
$$n\, xy^n  = L_{n-1,1}(y,xy) - (n-1)\, L_{n-1,1}(y,yx) + L_{n-2,1,1}(y,x,y^2)$$ 
in $\FF\X$.
\end{lemma}
\begin{proof} For short, we write $x$ for $x_1$ and $y$ for $x_{n+1}$. In $\ZZ\X$ holds 
\begin{eq}\label{eq_sec4_1}
\begin{array}{rl}
L_{n-1,1}(x,xy) & = \sum_{i=0}^{n-1} x^i xy x^{n-i-1} \\
&=  x\sum_{i=0}^{n-1} x^i y x^{n-i-1} = x L_{n-1,1}(x,y).\\
\end{array}
\end{eq}%
Note that two elements of $\ZZ\X$ are equal to each other if and only if its homogeneous components with respect to the multidegree are equal to each other. Thus, making the  substitution $x\to x_1+\cdots+x_n$ in the above equality and taking the homogeneous component of multidegree $1^{n+1}$ (i.e. of degree one in each of letters $x_1,\ldots,x_n,y$), we obtain   
$$\sum_{i=1}^n L_{1^n}(x_1,\ldots,x_{i-1},x_{i+1},\ldots,x_n,x_i y) = 
\sum_{i=1}^n x_i L_{1^n}(x_1,\ldots,x_{i-1},x_{i+1},\ldots,x_n,y).
$$%
We make substitutions $x_2\to y,\ldots,x_n\to y$ and recall that $x_1=x$. Thus in $\ZZ\X$ we have
\begin{eq}\label{eq_sec4_2}
\begin{array}{rl}
& L_{1^n}(\underbrace{y,\ldots,y}_{n-1},xy) +  (n-1)\, L_{1^n}(x,\underbrace{y,\ldots,y}_{n-2},y^2) \\
=&x L_{1^n}(y,\ldots,y,y) +  (n-1)\, y L_{1^n}(x,\underbrace{y,\ldots,y}_{n-2},y).\\
\end{array}
\end{eq}%
By~\Ref{eq_sec4_1}, the second summand of the right hand side of~\Ref{eq_sec4_2} is equal to $(n-2)!(n-1)^2\,L_{n-1,1}(y,yx)$. Hence
\begin{eq}\label{eq_sec4_3}
\begin{array}{rl}
& (n-1)!\,L_{n-1,1}(y,xy) +  (n-1)!\, L_{n-2,1,1}(y,x,y^2) \\
=&n!\,x y^n +  (n-2)!(n-1)^2\,L_{n-1,1}(y,yx)\\
\end{array}
\end{eq}%
in $\ZZ\X$. Obviously, we can divide~\Ref{eq_sec4_3} by $(n-1)!$. As the result, we obtain that the required equality hold over $\ZZ$ and, therefore, it holds over $\FF$.
\end{proof}

In the rest of this section we assume that $\#\FF=n-1$ and $n\geq3$. By Theorem~\ref{theo1} (see also Example~\ref{ex1}), in $N_{n,d}$ we have
\begin{eq}\label{eq_sec4_4}
L_{1,n-1}(x,y) + L_{n-1,1}(x,y)=0 \text{ and }
\end{eq}
\vspace{-0.5cm}
\begin{eq}\label{eq_sec4_4prime}
L_{\un{\theta}}(x_1,\ldots,x_r)=0, \text{ where } \#\un{\theta}=r,\; |\un{\theta}|=n \text{ and } \un{\theta}\not\in\{(1,n-1),(n-1,1)\}. 
\end{eq}

\begin{lemma}\label{lemma_sec4_2}
In case $\#\FF=n-1$ the following equalities hold in $N_{n,d}$:
\begin{enumerate}
\item[(a)] $x L_{n-1,1}(y,z) = L_{n-1,1}(y,xz)$; 

\item[(b)] $L_{n-1,1}(x,y)z = L_{n-1,1}(x,yz)$; 

\item[(c)] $L_{n-1,1}(x,yz) = -L_{n-1,1}(z,yx)$.
\end{enumerate}
\end{lemma}
\begin{proof}We can assume that $x=x_1$, $y=x_2$, and $z=x_3$. By Lemma~\ref{lemma_sec4_1}, 
$$xy^n  = L_{n-1,1}(y,xy) + L_{n-2,1,1}(y,x,y^2)$$ 
in $\FF\X$. Note that two elements of $\FF\X$ are equal to each other if and only if its homogeneous components with respect to the multidegree are equal to each other. Thus, making the  substitution $y\to y+z$ in the above equality and taking the homogeneous component of multidegree $(1,n-1,1)$ (i.e. of degree one in $x$, $z$ and degree $n-1$ in $y$), we obtain
$$
\begin{array}{rl}
x L_{n-1,1}(y,z) & = L_{n-1,1}(y,xz) + L_{n-2,1,1}(y,z,xy) \\
& + L_{n-2,1,1}(y,x,yz)+ L_{n-2,1,1}(y,x,zy) + L_{n-3,1,1,1}(y,z,x,y^2)\\
\end{array}
$$%
in $\FF\X$. Applying~\Ref{eq_sec4_4prime}, we complete the proof of part~(a) of the lemma. 
Part~(b) follows from Remark~\ref{remark_inversion}.

By~\Ref{eq_sec4_4}, $x L_{n-1,1}(y,z) + x L_{n-1,1}(z,y)=0$ in $N_{n,d}$. Thus part~(a) of the lemma implies part~(c).
\end{proof}

\begin{lemma}\label{lemma_sec4_3}
In case $\#\FF=n-1$ we have $uL_{n-1,1}(a,b)v\approx0$ in $N_{n,d}$ for all $a,b\in\X_d$, $u,v\in\X_d^{\#}$ with $\deg(ubv)\geq2$.
\end{lemma}
\begin{proof}
We work in $N_{n,d}$. Note that the equalities from parts~(a) and~(b) of Lemma~\ref{lemma_sec4_2} are homogeneous with respect to the degree. Applying these equalities, without loss of generality we can assume that $\deg(b)\geq2$. We set $f=L_{n-1,1}(a,b)$ and $b=cx$ for a letter $x$ and a word $c\in\X_d$. Part~(c) of Lemma~\ref{lemma_sec4_2} and~\Ref{eq_sec4_4} imply that $f=-L_{n-1,1}(x,ca) = L_{n-1,1}(ca,x)$. The inequality $\deg(L_{n-1,1}(ca,x))>\deg(f)$ completes the proof. 
\end{proof}

Now we can proof Theorem~\ref{theo_sec4}.

\begin{proof} We set $p=\Char{\FF}=\Char{\LL}$. Consider a $k>C=C_{n,d}^{\LL}$. Let $u$ be a word of $\X_d$ with $\deg(u)=k$.  We set $u=vx$ for a letter $x$ and a word $v\in\X_d$ with $\deg(v)\geq C$. Since the ideal of relations for $N_{n,d}^{\LL}$ is generated by elements with coefficients $0$ and $1$ (see Remark~\ref{remark0}), we have 
$$v=\sum_i \al_i L_{\un{\theta}_i}(\un{a}_i)$$
in $\FF\X$ for $\al_i\in\FF_p$, $|\un{\theta}_i|=n$, and $\un{a}_i=(a_{i1},\ldots,a_{ir})$, where $r=\#\un{\theta}_i$ and $a_{i1},\ldots,a_{ir}\in\X_d$. Lemma~\ref{lemma_sec4_3} together with relation~\Ref{eq_sec4_4prime} imply $u=vx\approx 0$ in $N_{n,d}^{\FF}$. Thus if $w=0$ in $N_{n,d}^{\FF}$ for all $w\in\X_d$ with $\deg(w)=k+1$, then $w=0$ in $N_{n,d}^{\FF}$ for all $w\in\X_d$ with $\deg(w)=k$. Hence $C_{n,d}^{\FF}\leq C_{n,d}^{\LL}+1$. The inequality $C_{n,d}^{\LL}\leq C_{n,d}^{\FF}$ follows from part~2 of Corollary~\ref{cor1}.
\end{proof}

Corollary~\ref{cor1}, Theorem~\ref{theo_sec4} together with Theorem~5.1 from~\cite{Lopatin_Nnd} imply the following corollary.

\begin{cor}\label{cor_n4}
For an arbitrary field $\FF$ of characteristic $p$ and $d\geq2$ we have
\begin{enumerate}
\item[$\bullet$] $C_{4,d}=10$, if $p = 0$;

\item[$\bullet$] $3d < C_{4,d}$, if $p = 2$;

\item[$\bullet$] $3d + 1\leq C_{4,d}\leq 3d + 5$, if $p = 3$;

\item[$\bullet$] $10\leq C_{4,d}\leq 13$, if $p > 3$.
\end{enumerate}
\end{cor}

\section*{Acknowledgements}
This paper was supported by FAPESP No.~2011/51047-1. 


\end{document}